\documentclass{article}
\usepackage{graphicx} 
\usepackage{geometry}
\usepackage{hyperref}

\usepackage{tikz}
\usetikzlibrary{decorations.markings,arrows}
\usetikzlibrary{calc}
\usetikzlibrary{math}
\usetikzlibrary {arrows.meta}

\usepackage{url}
\usepackage{amsmath}
\usepackage{amssymb}
\usepackage{amsfonts}
\usepackage{amsthm}
\newtheorem{thm}{Theorem}
\newtheorem{lem}{Lemma}
\newtheorem{cor}{Corollary}

\NewDocumentCommand{\rot}{O{90} O{1em} m}{\makebox[#2][l]{\rotatebox{#1}{#3}}}%

\title{Lower bounds and integrality gaps in simplicial decomposition}
\author{Matthew Ellison\\
Department of Mathematics\\
Dartmouth College\\
\texttt{matthew.ellison.gr@dartmouth.edu}}
\date{}

\begin{document}

\maketitle

\abstract
Let $\mathcal{K}$ be a finite pure simplicial $d$-complex, with oriented facets $\{F_i\}$, which is boundaryless in the sense that $\sum\partial F_i=0$. We call such a $\mathcal{K}$ an \textit{admissible $d$-complex}. Given an admissible $d$-complex, one can ask for the smallest collection $\{T_i\}$ of oriented $(d+1)$-simplices on the vertices of $\mathcal{K}$ which decomposes $\mathcal{K}$ in the sense that $\sum \partial T_i = \mathcal{K}$. Let the minimum size of such a collection be $V_\mathbb{Z}(\mathcal{K})$, and let $V_\mathbb{Q}(\mathcal{K})$ be the relaxed analog where fractional $(d+1)$-simplices may be used. We explain how these quantities may be computed via integer and linear programming, and show how lower bounds may be obtained by exploiting LP-duality. We then prove that $V_\mathbb{Q}$ and $V_\mathbb{Z}$ are both additive under disjoint union and connected sum along a $d$-simplex. The remainder of the paper explores integrality gaps between $V_\mathbb{Z}$ and $V_\mathbb{Q}$ in dimension 1, where we share what we believe is the simplest admissible complex with an integrality gap; and in dimension 2, where we collect some results on integrality gaps for triangulations of the 2-sphere for \cite{doyle2023extending}, to which this is a companion paper.

\section{Introduction}
Let $\mathcal{K}$ be a finite pure simplicial $d$-complex, with oriented facets $F_i$, which is boundaryless in the sense that $\sum\partial F_i=0$. We call such a $\mathcal{K}$ an \textit{admissible $d$-complex}. For example, an admissible $1$-complex is a directed graph where each node has equal in-degree and out-degree; and admissible 2-complexes include the boundaries of the platonic solids (with a natural orientation).

For admissible $d$-complex $\mathcal{K}$, and field $\mathbb{F}$ equal to $\mathbb{Z}$ or $\mathbb{Q}$, let $C_n(\mathcal{K}; \mathbb{F})$ denote the chains of oriented $n$-simplices on the vertices of $\mathcal{K}$ with coefficients in $\mathbb{F}$. For $\alpha$ in $C_n(\mathcal{K};\mathbb{F})$, let $|\alpha|_1$ denote the 1-norm of $\alpha$, i.e. the sum of the absolute values of its coefficients. The main definition is the following notion of volume. Let $V_\mathbb{F}(\mathcal{K})$ be the smallest 1-norm of a chain $\alpha$ in $C_{d+1}(\mathcal{K};\mathbb{F})$ which decomposes $\mathcal{K}$ in the sense that $\partial \alpha=\mathcal{K}$. Informally, $V_\mathbb{Z}$ is the fewest number of $(d+1)$-simplices needed to decompose $\mathcal{K}$, and $V_\mathbb{Q}$ is the fewest number of $(d+1)$-simplices needed to decompose $\mathcal{K}$ when `fractional simplices' are allowed.

In this paper we investigate properties of $V_\mathbb{Z}$ and $V_\mathbb{Q}$, and the integrality gap that may appear between them. In Section \ref{example_section}, $V_\mathbb{Z}$ and $V_\mathbb{Q}$ are computed for two simple examples.
 In Section \ref{lp_ip_section}, we explain how $V_\mathbb{Z}$ and $V_\mathbb{Q}$ can be computed through integer and linear programming, and Section \ref{explicit_lp_ip_section} gives a worked example to illustrate this idea. In Section \ref{sec_flux_lemmas}, we take a closer look at duality, formalizing the `flux method' of Section \ref{example_section} and proving some technical lemmas needed for Sections \ref{additivity_seciton} and \ref{additivity_section_z}, where we prove the additivity of $V_\mathbb{Q}$ and $V_\mathbb{Z}$ under disjoint union and $d$-simplex connected sum. The remaining sections take a closer look at the integrality gap between $V_\mathbb{Z}$ and $V_\mathbb{Q}$. In Section \ref{sec_ig_1d}, we share an admissable 1-complex exhibiting an integrality gap, which we believe is the simplest case where $V_\mathbb{Z}\neq V_\mathbb{Q}$. In Section \ref{sec_ig_2d}, we collect some results on integrality gaps for triangulations of the sphere for use in the companion paper \cite{doyle2023extending}.

The idea of $V_\mathbb{Z}$ and $V_\mathbb{Q}$ trace back to Gromov, who first defined the \textit{simplicial norm} $g(\alpha)$ of a homology class (of any topological space) in \cite{gromov}, which extends to the notion of a manifold's \textit{simplicial volume} via its fundamental class. $V_\mathbb{Z}$ and $V_\mathbb{Q}$ fit into a simplicial analog of Gromov's singular simplicial theory. An essential difference is that while a manifold's simplical volume in Gromov's sense is a homotopy invariant, $V_\mathbb{Z}$ and $V_\mathbb{Q}$ depends strongly on the simplicial decomposition. Consider, for example, the shell of a tetrahedron versus the shell of an icosahedron.

In \cite{stt}, Sleater, Thurston, and Tarjan, while investigating rotations in binary trees, and an equivalent notion of flip distance, had occasion to consider the fewest number of tetrahedra needed to geometrically decompose certain triangulations of the 2-sphere in 3-space\footnote{Flip distance remains an interesting topic of its own, see, for example, \cite{pournin2021strong}.}. The authors lower-bounded this number via a hyperbolic geometry argument, but noted, at the end of the paper, that it should be possible to follow a combinatorial approach ---- remarking that a simplicial relaxation of the geometric problem, essentially what we consider in this paper, is a linear problem. This thread was carried forward to some extent in \cite{kt}, where they essentially work to construct lower-bounds --- via the `flux method' --- for an infinite family of two-dimensional triangulations using ideas of max-flow.

More recently, in a collaboration \cite{doyle2023extending} with Zili Wang and Peter Doyle extending the results of \cite{stt}, we had the desire to actually compute $V_\mathbb{Z}$ for some triangulations of the sphere. This work provides some results for \cite{doyle2023extending}.

\section{Two examples}\label{example_section}
We begin with some hands-on, and relatively ad-hoc, computations of $V_\mathbb{Z}$ and $V_\mathbb{Q}$.\\

\subsection{Example 1: 6-cycle} Let $\mathcal{K}$ consist of the edges of a regular hexagon in the plane (oriented outward/counterclockwise). For $V_\mathbb{Z}$, we seek a minimum-size collection of 2-simplices $\{T_i\}$ on the vertices of $\mathcal{K}$ such that $\sum \partial T_i =\mathcal{K}$. It seems clear that we can't do better than a standard decomposition of the hexagon into four triangles (say, for example, coning to vertex), but how can we prove it?\\
\begin{figure}[h]
    \centering
    \tikzmath{\s = 3; \es = \s * (sqrt(3)/2); \les = \s / 2; \al =.92; \ah = 1.4; \at = (\al + \ah) / 2 + .05; \ta = 4; \tal = 9;}
\begin{tikzpicture}
    \foreach \x in {0,1,2,3,4,5}
        {
        \draw ({\s * cos(\x * 60)}, {\s * sin(\x * 60)}) -- ({\s * cos((\x + 1) * 60)}, {\s * sin((\x + 1) * 60)});
        
        \draw[->] ({\al * \es * cos((\x + 1/2)* 60)}, {\al * \es * sin((\x + 1/2) * 60)}) -- ({\ah * \es * cos((\x + 1/2)* 60)}, {\ah * \es * sin((\x + 1/2) * 60)});
        
        \node at ({\at * \es * cos((\x + 1/2)* 60 + \ta])}, {\at * \es * sin((\x + 1/2) * 60 + \ta)}) {1};
        
        \draw ({\s * cos(\x * 60)}, {\s * sin(\x * 60])}) -- ({\s * cos((\x + 2) * 60)}, {\s * sin((\x + 2) * 60)});

        \draw[->] ({\al * \les * cos(\x* 60)}, {\al * \les * sin(\x * 60)}) -- ({\ah * \les * cos(\x * 60)}, {\ah * \les * sin(\x * 60)});

        \node at ({\at * \les * cos(\x * 60 + \tal])}, {\at * \les * sin(\x * 60 + \tal)}) {\footnotesize 1/2};
        }
\end{tikzpicture}
    \caption{Fluxes to show at least 4 triangles are needed}
    \label{fig:hexflux}
\end{figure}
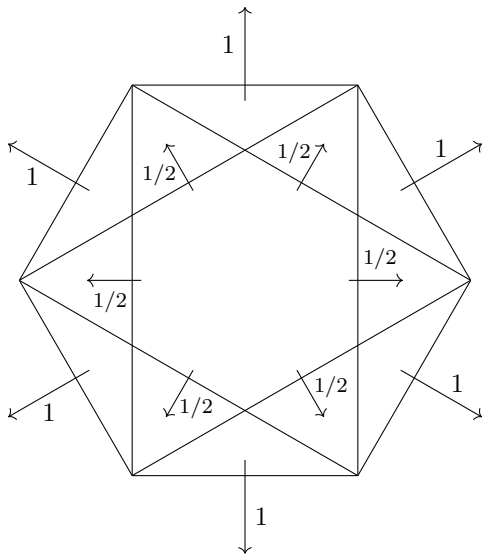
In Figure \ref{fig:hexflux}, `fluxes' have been placed across each edge between two hexagon vertices. The three unlabeled edges have zero flux. One can check that any triangle in a decomposition must have outward flux at most 3/2. On the other hand, the total outward flux in a decomposition must be equal to 6, the flux through the boundary. Thus a decomposition must have at least $6 / (3/2) = 4$ triangles, and so we have $V_\mathbb{Z}(\mathcal{K})\geq 4$. The same argument applies equally well to show $V_\mathbb{Q}(\mathcal{K})\geq 4$, and hence, since we can achieve $4$ with integral coefficients, we have $V_\mathbb{Z}=V_\mathbb{Q}=4$. To see the fluxes don't depend on our picture, note that they are really just weights on possible directed edges, which produce, via the boundary map, weights on 2-chains.

\subsection{Example 2: triangular bipyramid} Let $\mathcal{K}$ consist of the faces of a triangular bipyramid (oriented outward). For $V_\mathbb{Z}$, we seek a minimum-size collection of oriented 3-simplices $\{T_i\}$ on the vertices of $\mathcal{K}$ such that $\sum \partial T_i =\mathcal{K}$. We can clearly achieve 2 tetrahedra by taking the bottom and top pyramids. It's interesting to note, although we already did better, that one can also achieve 3 by taking meridianal wedges. Let's prove at least two are needed by the same flux approach.\\
\begin{figure}[h]
    \centering
    \tikzmath{\cx = 2.8; \cy = 1.6;  \ux = 2.5; \uy = 3.5; \dx = 2.5; \dy = -2.5; \ter = .6; \tix = 1.9; \tiy = 1.4; \der = .7; \dix = 3.1; \diy = -.4;}
\begin{tikzpicture}
    \coordinate (A) at (0,0);
    \coordinate (B) at (4,0);
    \coordinate (C) at (\cx, \cy);
    \draw (A) -- (B);
    \draw [dashed] (A) -- (C);
    \draw [dashed] (B) -- (C);
    \coordinate (U) at (\ux, \uy);
    \draw (A) -- (U);
    \draw (B) -- (U);
    \draw [dashed] (C) -- (U);
    \coordinate (D) at (\dx, \dy);
    \draw (A) -- (D);
    \draw (B) -- (D);
    \draw [dashed] (C) -- (D);
    
    \coordinate (TLI) at ({\ux * \ter}, {\uy * \ter});
    \coordinate(TLD) at (\tix, \tiy);
    \coordinate(TLU) at ({\ux * \ter + (\ux * \ter - \tix)}, {\uy * \ter + (\uy * \ter - \tiy});
    \draw [dashed] (TLI) -- (TLD);
    \draw[->] (TLI) -- (TLU);
    \node at ({\ux * \ter + (\ux * \ter - \tix) + .3}, {\uy * \ter + (\uy * \ter - \tiy) + .1}) {1};
    \draw [draw=none, fill=green, fill opacity=0.2] (0,0) -- (\cx,\cy) -- (\ux,\uy) -- cycle;
    \coordinate (TLI) at ({\ux * \ter}, {\uy * \ter});
    \coordinate(TLD) at (\tix, \tiy);
    \coordinate(TLU) at ({\ux * \ter + (\ux * \ter - \tix)}, {\uy * \ter + (\uy * \ter - \tiy});
    \draw [dashed] (TLI) -- (TLD);
    \draw[->] (TLI) -- (TLU);
    \draw[fill = black] (\tix,\tiy) circle (.3ex);
    \node at ({\ux * \ter + (\ux * \ter - \tix) + .3}, {\uy * \ter + (\uy * \ter - \tiy) + .1}) {1};
    \draw [draw=none, fill=green, fill opacity=0.2] (4,0) -- (\cx,\cy) -- (\dx,\dy) -- cycle;
    \coordinate (DLI) at ({\dx + \der * (4-\dx)}, {\dy + \der * (-\dy)});
    \coordinate(DLD) at (\dix, \diy);
    \coordinate(DLU) at ({\dx + \der * (4-\dx) + 1.2*(\dx + \der * (4-\dx) - \dix)}, {\dy + \der * (-\dy) + 1.2*(\dy + \der * (-\dy) - \diy});
    \draw [dashed] (DLI) -- (DLD);
    \draw[->] (DLI) -- (DLU);
    \draw[fill = black] (\dix,\diy) circle (.3ex);
    \node at ({\dx + \der * (4-\dx) + 1.2*(\dx + \der * (4-\dx) - \dix)-.1}, {\dy + \der * (-\dy) + 1.2*(\dy + \der * (-\dy) - \diy - .2}) {1};
\end{tikzpicture}
    \caption{Fluxes to show at least 2 tetrahedra are needed}
    \label{fig:bipyrflux}
\end{figure}
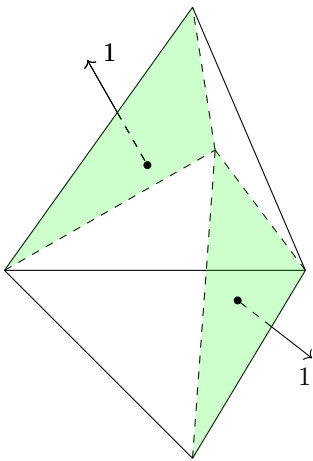
In the above diagram, we have assigned flux 1 through the two faces as shown, and no flux through any of the 8 other triangles one can form between the 5 vertices of the bipyramid. It's not hard to see that any 3-simplex can have outward flux at most 1. The whole shape, on the other hand, has outward flux 2, and so any decomposition must have at least $2/1=2$ tetrahedra, integral or fractional. Thus $V_\mathbb{Z}(\mathcal{K})=V_\mathbb{Q}(\mathcal{K})=2$.

\section{$V_\mathbb{Z}, V_\mathbb{Q}$ via integer programming, linear programming}\label{lp_ip_section}

Finding $V_\mathbb{Z}(\mathcal{K})$ can be cast as an integer linear programming problem. Each possible oriented $(d+1)$-simplex contributes its boundary, and we are looking for an integer linear combination equal to $\mathcal{K}$ where the number of $(d+1)$-simplices is minimized.

On the computational level, we take two primal variables for each $(d+2)$-tuple of vertices of $\mathcal{K}$ (one for each orientation), and we restrict these primal variables to non-negative integers. We have two dual variables for each $(d+1)$-tuple of vertices (one for each orientation), and the corresponding constraints are that for each oriented $d+1$-tuple we must have as many on the boundary of $\mathcal{K}$ as are accounted for by the choices of primal variables. We want to minimize the sum of the primal variables (i.e. the number of $(d+1)$-simplices chosen). See the next section for a fully worked-out example with the bipyramid. From the above discussion, we see that the linear program will have $2\binom{n}{d+2}$ variables and $2\binom{n}{d+1}$ constraints, where $n$ is the number of vertices in $\mathcal{K}$. Such an integer program can be solved by standard software libraries, for example SciPy's\cite{2020SciPy-NMeth}\texttt{optimize.linprog}, which supports linear programming over both $\mathbb{Z}$ (wrapping the HiGHS' `Branch and Price' \cite{highs}) and $\mathbb{Q}$. Solving over $\mathbb{Z}$ yields $V_\mathbb{Z}$ and solving over $\mathbb{Q}$ yields $V_\mathbb{Q}$.

Since the linear program for $V_\mathbb{Q}$ is a relaxation of the integer program from $V_\mathbb{Z}$, we must have $V_\mathbb{Q}(\mathcal{K})\leq V_\mathbb{Z}(\mathcal{K})$, and, if this inequality is strict, our decomposition of $\mathcal{K}$ can improve through the use of `fractional simplices'. Note that this really does occur, see Sections \ref{sec_ig_1d} and \ref{sec_ig_2d} for examples.

It's interesting to note that the dual of the linear program for $V_\mathbb{Q}(\mathcal{K})$ amounts to finding an optimal choice of `fluxes' like we did in Section \ref{example_section} --- where by optimal we mean the best implied lower bound on $V_\mathbb{Z}$ and $V_\mathbb{Q}$. By linear programming duality, this flux method is able to obtain the lower bound of  $V_\mathbb{Q}$ but not better. This duality is explicitly presented in Section \ref{explicit_lp_ip_section}, and we collect some theoretical results in Section \ref{sec_flux_lemmas}.

\section{Explicit example of the programs for $V_\mathbb{Q}$ and $V_\mathbb{Z}$}\label{explicit_lp_ip_section}
Consider the bipyramid example from Section \ref{example_section}, where we have now labeled the vertices in Figure \ref{fig:bipyrlab}.
\begin{figure}[h]
    \centering
    \tikzmath{\cx = 2.8; \cy = 1.6;  \ux = 2.5; \uy = 3.5; \dx = 2.5; \dy = -2.5; \ter = .6; \tix = 1.9; \tiy = 1.4; \der = .7; \dix = 3.1; \diy = -.4;}
\begin{tikzpicture}
    \coordinate (A) at (0,0);
    \coordinate (B) at (4,0);
    \coordinate (C) at (\cx, \cy);
    \draw (A) -- (B);
    \draw [dashed] (A) -- (C);
    \draw [dashed] (B) -- (C);
    \coordinate (U) at (\ux, \uy);
    \draw (A) -- (U);
    \draw (B) -- (U);
    \draw [dashed] (C) -- (U);
    \coordinate (D) at (\dx, \dy);
    \draw (A) -- (D);
    \draw (B) -- (D);
    \draw [dashed] (C) -- (D);

    \node at (0 - .3,0) {B};
    \node at (4 + .3,0) {D};
    \node at (\ux, \uy + .3) {A};
    \node at (\dx, \dy - .3) {C};
    \node at (\cx -.3, \cy + .1) {E};
\end{tikzpicture}
    \caption{Labelled bipyramid}
    \label{fig:bipyrlab}
\end{figure}
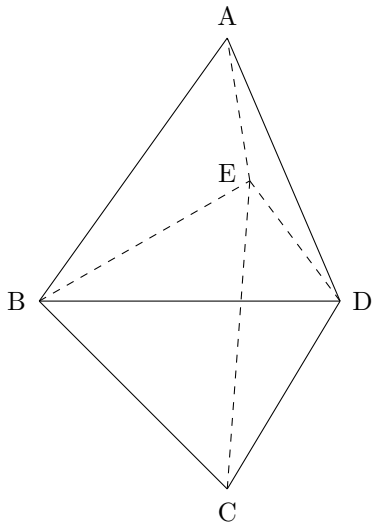
The associated simplicial complex $\mathcal{K}$ has facets  $ABD, AEB, ADE, BCD, CED, CBE$, and the associated linear program for $V_\mathbb{Q}(\mathcal{K})$ may be expressed in tableau form\footnote{See, for example, \cite{vas} p. 77.} as follows:\\
\begin{tabular}{c c c c c c c c c c c || c}
DEBA & DEAB & DBEC & DBCE & DBAC & DBCA & EDAC & EDCA & BEAC & BECA & 1 & \\
\hline
\hline
1 & -1 & & & 1& -1& & & & & -1& ABD\\
\hline
-1 & 1& & & -1& 1& & & & & 1& ADB\\
\hline
1 & -1 & & & & & 1& -1& & & -1& ADE\\
\hline
-1 & 1& & & & & -1& 1& & & 1& AED\\
\hline
1 & -1& & & & & & & 1& -1& -1& AEB\\
\hline
-1 & 1& & & & & & & -1& 1& 1& ABE\\
\hline
& & 1& -1& 1& -1& & & & & -1& BCD\\
\hline
& & -1& 1& -1& 1& & & & & 1& ADC\\
\hline
& & 1& -1& & & 1& -1& & & -1& CED\\
\hline
& & -1& 1& & & -1& 1& & & 1& CDE\\
\hline
& & 1& -1& & & & & 1& -1& -1& ECB\\
\hline
& & -1& 1& & & & & -1& 1& 1& EBC\\
\hline
1 & -1& -1& 1& & & & & & & & BED\\
\hline
-1 & 1& 1& -1& & & & & & & & BDE\\
\hline
& & & & -1& 1& & & 1& -1& & ABC\\
\hline
& & & & 1& -1& & & -1& 1& & ACB\\
\hline
& & & & 1& -1& -1& 1& & & & ADC\\
\hline
& & & & -1& 1& 1& -1& & & & ACD\\
\hline
& & & & & & 1& -1& -1& 1& & AEC\\
\hline
& & & & & & -1& 1& 1& -1& & ACE\\
\hline
1 & 1 & 1 & 1 & 1& 1 & 1 & 1 & 1 & 1 & & $\rightarrow\text{min}$\\
\end{tabular}
In this tableau, the first two rows, for example, enforce the constraints that 
\begin{align*}
    \#DEBA - \#DEAB + \#DBAC - \#DBCA -1 &\geq 0\\
    -\#DEBA + \#DEAB - \#DBAC + \#DBCA +1 &\geq 0
\end{align*}
The first enforces \textit{at least 1} $ABD$ in the boundary of our decomposition. The second enforces at least -1 $ADB$ in the boundary of decomposition, which is equivalent to there being \textit{at most 1} $ABD$ in the boundary of our decomposition. Together they enforce exactly 1 $ABD$ in our boundary. So all the rows but the last constrain our decomposition to have the correct boundary, and the last expresses that we want as few tetrahedra as possible.
The optimal value of this program over the integers is $V_\mathbb{Z}$ and the optimal value over rationals is $V_\mathbb{Q}$, in this case both are equal to $2$. Note that the size of the program is $\approx2\binom{n}{4}\times 2\binom{n}{3}$, where $n$ is the number of vertices.

It is instructive to look at the dual linear program, which we may obtain by a simple transposition of the tableau with sign changes (\cite{vas} p. 137):\\

\begin{tabular}{c |c |c |c |c |c |c| c |c |c |c |c |c |c |c |c |c |c |c |c |c|| c}
\rot{ABD$^*$}& \rot{ADB$^*$} & \rot{ADE$^*$} & \rot{AED$^*$} & \rot{AEB$^*$}& \rot{ABE$^*$}& \rot{BCD$^*$}& \rot{ADC$^*$}& \rot{CED$^*$}& \rot{CDE$^*$}& \rot{ECB$^*$}& \rot{EBC$^*$}& \rot{BED$^*$}& \rot{BDE$^*$}& \rot{ABC$^*$}& \rot{ACB$^*$}& \rot{ADC$^*$}& \rot{ACD$^*$}& \rot{AEC$^*$}& \rot{ACE$^*$}&1 &\\
\hline\hline
-1 &1 &-1 &1 &-1 &1 & & & & & & &-1 &1 & & & & & & &1 & DEBA$^*$\\
1 & -1& 1 & -1& 1& -1& & & & & & & 1& -1& & & & & & & 1 & DEAB$^*$\\
 & & & & & & 1& 1& -1& 1& -1& 1& 1& -1& & & & & & & 1 &DBEC$^*$\\
 & & & & & & 1& -1& 1& -1& 1& -1& -1& 1& & & & & & & 1 &  DBCE$^*$ \\
 -1& 1& & & & & -1& 1& & & & & & & 1& -1& -1& 1& & & 1 &  DBAC$^*$ \\
 1& -1& & & & & 1& -1& & & & & & & -1& 1& 1& -1& & & 1 &  DBCA$^*$ \\
 & & -1& 1& & & & & -1& 1& & & & & & & 1& -1& -1& 1& 1 &  EDAC$^*$ \\
 & & 1& -1& & & & & 1& -1& & & & & & & -1& 1& 1& -1& 1 &  EDCA$^*$ \\
 & & & & -1& 1& & & & & -1& 1& & & -1& 1& & & 1& -1& 1 &  BEAC$^*$ \\
 & & & & 1& -1& & & & & 1& -1& & & 1& -1& & & -1& 1& 1 &  BECA$^*$ \\
 1 & -1& 1& -1& 1& -1& 1& -1& 1& -1& 1& -1& 1& & & & & & & & & 
 $\rightarrow\max$\\
\end{tabular}
How can we interpret this dual program? We are assigning a flux through each triangle in such a way that (from the constraints) the flux out of each possible tetrahedron is at most 1, and we are doing this to maximize the total flux out out of our bipyramid --- asking, in other words, for the best lower bound from the flux technique as in Section $\ref{example_section}$. By linear programming duality the maximum value of this dual program (i.e. the best flux lower bound) is exactly equal to $V_\mathbb{Q}=2$.

\section{Duality and Flux Functions}\label{sec_flux_lemmas}
We collect a few lemmas relating to duality in this section, mainly for use in Section \ref{additivity_seciton}.

\begin{lem}[Flux Lower Bound]\label{flux_lb_lemma}
    Let $\mathcal{K}$ be an admissible $d$-complex, and let $S_d$ be the set of oriented $d$-simplices on the vertices of $\mathcal{K}$. Suppose we have a real-valued `flux function' $\phi$ on the $S_k$ such that:\\
    \begin{itemize}
        \item For any $T\in S_d$, with $T'$ its oppositely oriented version, 
        $$\phi(T)=-\phi(T').$$
        \item For any $T$ in $S_{d+1}$, with boundary $B_1,...,B_{d+2}$, $$\phi(B_1)+...+\phi(B_{d+2})\leq 1.$$
        i.e. every candidate $(d+1)$-simplex has outward flux at most 1.
    \end{itemize}
    Let $F_1, ..., F_n$ be the oriented facets of $\mathcal{K}$. Then 
        $$\phi(F_1)+...+\phi(F_n)\leq V_\mathbb{Q}(\mathcal{K}).$$
\end{lem}
\begin{proof}
Immediate from LP duality.
\end{proof}
\begin{lem}[Optimal flux function]\label{flux_lemma}
    Let $\mathcal{K}$ be an admissible $d$-complex, and let $S_d$ be the set of oriented $d$-simplices on the vertices of $\mathcal{K}$. Then there exists a real-valued `flux function' $\phi$ on $S_d$ such that:\\
    \begin{itemize}
        \item For any $T\in S_d$, with $T'$ its oppositely oriented version, 
        $$\phi(T)=-\phi(T').$$
        \item For any $T$ in $S_{d+1}$, with boundary $B_1,...,B_{d+2}$, $$\phi(B_1)+...+\phi(B_{d+2})\leq 1.$$
        i.e. every candidate $(d+1)$-simplex has outward flux at most 1.
        \item Let $F_1, ..., F_n$ be the oriented facets of $\mathcal{K}$. Then 
        $$\phi(F_1)+...+\phi(F_n)=V_\mathbb{Q}(\mathcal{K}).$$
    \end{itemize}
    We will call such a $\phi$ an \underline{optimal flux function for $\mathcal{K}$}.
\end{lem}
\begin{proof}
Immediate from LP duality.
\end{proof}
We will make also make use of the following strengthened form of Lemma \ref{flux_lemma}.
\begin{lem}[Optimal flux function centered at $w$]\label{flux_lemma_strong}
    Let $\mathcal{K}$ be an admissible $d$-complex, and let $S_d$ be the set of oriented $d$-simplices on the vertices of $\mathcal{K}$. Fix a vertex $v$ of $\mathcal{K}$. Then there exists a real-valued `flux function' $\phi$ on $S_d$ such that:\\
    \begin{itemize}
        \item For any $T\in S_d$, with $T'$ its oppositely oriented version, 
        $$\phi(T)=-\phi(T').$$
        \item For any $T$ in $S_{d+1}$, with boundary $B_1,...,B_{d+2}$, $$\phi(B_1)+...+\phi(B_{d+2})\leq 1.$$
        i.e. every candidate $(d+1)$-simplex has outward flux at most 1.
        \item Let $F_1, ..., F_n$ be the oriented facets of $\mathcal{K}$. Then 
        $$\phi(F_1)+...+\phi(F_n)=V_\mathbb{Q}(\mathcal{K}).$$
        \item $\phi(T)\in [-1,1]$ for all $T$ in $S_d$.
        \item $\phi(T) = 0$ for any $T\in S_d$ with $w$ as vertex. 
    \end{itemize}
    We will call such a $\phi$ an \underline{optimal flux function for $\mathcal{K}$ centered at $w$}.
\end{lem}
\begin{proof}
Let $\phi_0$ be an optimal flux function for $\mathcal{K}$, as guaranteed by Lemma \ref{flux_lemma}. Then we have a function $d\phi_0$ on $S_{d+1}$, induced by $\phi$ via the boundary operator. We obtain our $\phi$ by `integrating' $d\phi_0$ with respect to $v$. In detail, for $T\in S_k$, define $\phi(T)=d\phi_0(\{w\}\wedge T)$, where by $\{w\}\wedge T$ we mean, if $T=[v_0,...,v_{d+1}]$, the oriented simplex $[w,v_0,...,v_{d+1}]$. If $w$ in $T$ already, we take $\phi(T)=0$.\\
It remains to verify that $\phi$ has the required properties, which we check in the order they are listed:
\begin{itemize}
    \item If $w\in T$ then $\phi(T)=-\phi(T')=0$. Otherwise, we have $$\phi(T)=d\phi_0(\{w\}\wedge T)=\sum \phi_0(\partial_i(\{w\}\wedge T))=\sum-\phi_0(\partial_i(\{w\}\wedge T'))=-\phi(T').$$
    \item For ease of notation, we extend $\phi, \phi_0, d\phi_0,\wedge$ to chains by linearity for the remainder of the proof. 
    $$\phi(B_1)+...+\phi(B_{d+2})=\phi( \partial T )= d\phi_0(\{w\}\wedge\partial T)=\phi_0(\partial(\{w\}\wedge\partial T))=\phi_0(\partial T - \{w\}\wedge\partial\partial T)= \phi_0(B_1)+...+\phi_0(B_{d+2})\leq 1.$$
    \item Reasoning as in the previous step yields 
    $$\phi(F_1)+...+\phi(F_n)=\phi(\partial K)=...=\phi_0(F_1)+...+\phi_0(F_n)=V_\mathbb{Q}(\mathcal{K}).$$
    \item $\phi(T)=d\phi_0(\{w\}\wedge T)=\phi_0(\partial (\{w\}\wedge T))$. By properties of $\phi_0$ this is $\leq 1$, and, since the same holds for $T$'s orientation reversal $T'$, $\phi(T)=-\phi(T')$ implies that $\phi(T)\geq -1$ as well.
    \item This holds by construction.
\end{itemize}
\end{proof}

\section{Additivity of $V_\mathbb{Q}$}\label{additivity_seciton}
We establish some results on the additivity of $V_\mathbb{Q}$ for admissible complexes of the same dimension.
\begin{thm}\label{thm: disjoint union}
    Let $\mathcal{K}$, $\mathcal{K}'$ be admissible $d$-complexes, and let $\mathcal{K}\sqcup\mathcal{K}'$ be their disjoint union. Then $$V_\mathbb{Q}(\mathcal{K}\sqcup\mathcal{K}') = V_\mathbb{Q}(\mathcal{K})+V_\mathbb{Q}(\mathcal{K}').$$
\end{thm}
\begin{proof}
Clearly $V_\mathbb{Q}(\mathcal{K}\sqcup\mathcal{K}') \leq V_\mathbb{Q}(\mathcal{K})+V_\mathbb{Q}(\mathcal{K}')$, since we can combine optimal decompositions of each part. For the lower bound, we proceed via Lemma 1. Let $w,w'$ be arbitrary vertices of $\mathcal{K}$, $\mathcal{K'}$, and let $\phi$, $\phi'$ be optimal flux functions centered at $w,w'$ as guaranteed by Lemma \ref{flux_lemma_strong}. We define a flux function $\phi^*$ for $\mathcal{K}\sqcup\mathcal{K}$ as follows. For $T$ an oriented $(d+1)$-simplex on the vertices of $\mathcal{K}\sqcup\mathcal{K}$, we define\\
$$\phi^*(T)=\begin{cases}
    \phi(T) & \text{if the vertices of $T$ are all in $\mathcal{K}$,}\\
    \phi'(T) & \text{if the vertices of $T$ are all in $\mathcal{K}$',}\\
     0 & \text{else.}\\
\end{cases}$$
Next we verify the conditions of Lemma \ref{flux_lb_lemma}. The first condition on $T$ and its oppositely-oriented version $T'$ is either inherited from $\phi$ or $\phi'$ or true since $\phi^*(T)=\phi^*(T')=0$. For the second condition, consider how the $d+2$ vertices of a given $T\in S_{d+1}$ are split between $\mathcal{K}$ and $\mathcal{K}'$. If $T$ is entirely in either, then the second condition is inherited from $\phi$ or $\phi'$. If $T$ is entirely in either except for one vertex in the other, then all facets of $\partial T$ are zero under $\phi^*$ except for the one entirely on one side. But this single facet is in $[-1,1]$ under $\phi^*$ (by the second-to-last property of Lemma \ref{flux_lemma_strong}), and so the second condition holds as well. Thus the conditions of Lemma \ref{flux_lb_lemma} are satisfied.

Let $F_1,...,F_n$ be the facets of $\mathcal{K}$ and let $F'_1,...,F'_{n'}$ be the facets of $\mathcal{K}'$. \ref{flux_lb_lemma}. Clearly $F_1,...,F_n,F'_1,...,F'_{n'}$ are the facets of $\mathcal{K}\sqcup\mathcal{K}'$, and by Lemma \ref{flux_lb_lemma} we have 
\begin{align*}
    \phi^*(F_1) + ... + \phi^*(F_n)+\phi^*(F'_1)+...+\phi^*(F'_{n'})&\leq V_\mathbb{Q}(\mathcal{K}\sqcup\mathcal{K}')\\
    V_\mathbb{Q}(\mathcal{K})+V_\mathbb{Q}(\mathcal{K}')&\leq V_\mathbb{Q}(\mathcal{K}\sqcup\mathcal{K}').
\end{align*}
This establishes the lower bound and completes the proof.
\end{proof}
The following proof, due to Peter Doyle, establishes that $V_\mathbb{Q}$ is also additive under ``$d$-simplex'' connected sum for admissible $d$-dimensional complexes.

\begin{thm}\label{thm: connected sum}
    Let $\mathcal{K},\mathcal{K}'$ be admissible $d$-complexes, $F=[v_0,...,v_k],F'=[v'_0,...,v'_k]$ be respective facets, and form the connected sum $\mathcal{K}\#\mathcal{K}'$ by starting with $\mathcal{K}\sqcup\mathcal{K}'$, then identifying the respective $v_i,v'_i$ and deleting $F,F'$. Then 
    $$V_\mathbb{Q}(\mathcal{K}\#\mathcal{K}') = V_\mathbb{Q}(\mathcal{K})+V_\mathbb{Q}(\mathcal{K}').$$
\end{thm}
\begin{proof}
    We proceed by the same strategy as in the proof of Theorem \ref{thm: disjoint union}.\\
    Clearly $V_\mathbb{Q}(\mathcal{K}\#\mathcal{K}') \leq V_\mathbb{Q}(\mathcal{K})+V_\mathbb{Q}(\mathcal{K}')$, since we can combine optimal decompositions of each part.\\
    For the lower bound, let $\phi,\phi'$ be optimal flux functions of $\mathcal{K}$,$\mathcal{K}$ centered at $v_0,v'_1$ respectively. Preparing to apply Lemma \ref{flux_lb_lemma}, we define a flux function $\phi^*$ on $\mathcal{K}\#\mathcal{K}'$ as follows:\\
    $$\phi^*(T)=\begin{cases}
    \phi(T) & \text{if the vertices of $T$ are all in $\mathcal{K}$ (including the $v_i/v'_i$),}\\
    \phi'(T) & \text{if the vertices of $T$ are all in $\mathcal{K}'$ (including the $v_i/v'_i$),}\\
     0 & \text{else.}\\
    \end{cases}$$
    Note that this is well-defined in the case that $T$ is on all the $v_i/v'_i$ because, since $\phi, \phi'$ based at $v_0,v'_1$, $\phi$ and $\phi'$ would both yield 0 in that case (see the last condition of Lemma \ref{flux_lemma_strong}). Now we verify that $\phi^*$ satisfies the conditions of Lemma \ref{flux_lb_lemma}. The first condition holds since it is either inherited from $\phi,\phi'$, or true since $0=-0$. For the second condition, we once again break into cases based on how the vertices of $T\in S_{d+1}$ are distributed between those of $\mathcal{K}$ and $\mathcal{K}'$:
    \begin{enumerate}
        \item \textit{All vertices in $\mathcal{K}$ (including the $v_i/v'_i$) or all vertices in $\mathcal{K}'$ (including the $v_i/v'_i$)}.\\
        In this case $\phi^*(\partial T)\leq 1$ is inherited from $\phi,\phi'$.
        \item \textit{At least 2 vertices not in $\mathcal{K}$ (including the $v_i/v'_i$), and at least 2 vertices not in $\mathcal{K}'$ (including the $v_i/v'_i$)}.\\
        In this case $\phi^*(\partial T)=0$, since every facet of the boundary of $T$ is 0 under $\phi^*$ by definition. Thus the condition holds.
        \item \textit{All but one vertex in either $\mathcal{K}$ (including the $v_i/v'_i$) or $\mathcal{K'}$ (including the $v_i/v'_i$)}.\\
        Let $w$ be the isolated vertex and $A$ the set of the remaining $d+1$ vertices of $T$. We break into cases based on how many of the $v_i/v'_i$ are included in $A$. If there are fewer than $d$, then the only facet of $\partial T$ which is nonzero under $\phi^*$ is the one on the vertices of $A$. The contribution of this facet must be in the range $[-1,1]$, and so the condition holds.\\
        If $d+1$ of the $v_i/v'_i$ are included in $A$ (i.e. $A=\{v_i/v'_i\}$, then, since any facet of $\partial T$ including $v_0/v'_0$ or $v_1/v'_1$ is zero under $\phi^*$, all facets are zero under $\phi^*$ and so the claim holds.\\
        This leaves the case where exactly $d$ of the $v_i/v'_i$ are included in $A$, and so we have two isolated vertices, $w$ from before and $w'$ in the other complex. The two candidates for non-zero value under $\phi$ are $w$ and the center together with $w'$ and the center. But (at least) one of $v_0/v'_0$ or $v_1/v'_1$ must be included in these two, and so we have at most one facet of $\partial T$ which is nonzero under $\phi^*$. Thus the claim holds.   
    \end{enumerate}
    Thus $\phi^*$ satisfies the hypotheses of Lemma \ref{flux_lb_lemma}.
    
    Let $F,F_1,...,F_{n-1}$ be the facets of $\mathcal{K}$ and let $F',F'_1,...,F'_{n'-1}$ be the facets of $\mathcal{K}'$. \ref{flux_lb_lemma}. Clearly $F_1,...,F_{n-1},F'_1,...,F'_{n'-1}$ are the facets of $\mathcal{K}\#\mathcal{K}'$, and by Lemma \ref{flux_lb_lemma} we have 
    \begin{align*}
    \phi^*(F_1) + ... + \phi^*(F_{n-1})+\phi^*(F'_1)+...+\phi^*(F'_{n'-1})&\leq V_\mathbb{Q}(\mathcal{K}\#\mathcal{K}')\\
    \intertext{Since $F$ and $F'$ are both zero under $\phi^*$, however, we have}
    \phi^*(F)+\phi^*(F_1) + ... + \phi^*(F_{n-1})+\phi^*(F')+\phi^*(F'_1)+...+\phi^*(F'_{n'-1})&\leq V_\mathbb{Q}(\mathcal{K}\#\mathcal{K}')\\
    V_\mathbb{Q}(\mathcal{K})+V_\mathbb{Q}(\mathcal{K}')&\leq V_\mathbb{Q}(\mathcal{K}\#\mathcal{K}').
\end{align*}
    This establishes the lower bound and completes the proof.
\end{proof}

\section{Additivity of $V_\mathbb{Z}$}\label{additivity_section_z}
We establish some results on the additivity of $V_\mathbb{Z}$ for admissible complexes of the same dimension.
\begin{thm}\label{thm:Z disjoint union}
    Let $\mathcal{K}$, $\mathcal{K}'$ be admissible $d$-complexes, and let $\mathcal{K}\sqcup\mathcal{K}'$ be their disjoint union. Then $$V_\mathbb{Z}(\mathcal{K}\sqcup\mathcal{K}') = V_\mathbb{Z}(\mathcal{K})+V_\mathbb{Z}(\mathcal{K}').$$
\end{thm}
\begin{proof}
Clearly we have $V_\mathbb{Z}(\mathcal{K}\sqcup\mathcal{K'})\leq V_\mathbb{Z}(\mathcal{K})+V_\mathbb{Z}(\mathcal{K}')$ since we can combine optimal decompositions of $\mathcal{K},\mathcal{K}'$ to obtain a candidate decomposition of $\mathcal{K}\sqcup\mathcal{K'}$.\\
For the reverse inclusion, suppose we have a $(d+1)$-chain $\alpha$ realizing $V_\mathbb{Z}(\mathcal{K}\sqcup\mathcal{K}')$. Fix vertices $v,v'$ in $\mathcal{K},\mathcal{K}'$ respectively. We obtain a new chain $\alpha'$ with no interaction between $\mathcal{K}$ and $\mathcal{K'}$ as follows:
\begin{enumerate}
    \item For each term $c_i T_i$ of $\alpha$, where the vertices of $V_i$ are either entirely in $\mathcal{K}$ or entirely in $\mathcal{K}'$, add this term to $\alpha'$ without modification.
    \item For each term $c_i T_i$ of $\alpha$, where the $T_i$ has all but one vertex in $\mathcal{K}$, contribute the term $c_i \tilde{T}_i$ to $\alpha'$ --- where $\tilde{T}$ is obtained by replacing the vertex of $T$ in $\mathcal{K'}$ by $v$.
    \item For each term $c_i T_i$ of $\alpha$, where the $T_i$ has all but one vertex in $\mathcal{K}'$, contribute the term $c_i \tilde{T}_i$ to $\alpha'$ --- where $\tilde{T}$ is obtained by replacing the vertex of $T$ in $\mathcal{K}$ by $v'$.
    \item Other terms of $\alpha$ are ignored, i.e. no contribution to $\alpha'$.
\end{enumerate}
Clearly we have $|\alpha'|_1\leq |\alpha|_1=V_\mathbb{Z}(\mathcal{K}\sqcup\mathcal{K}')$, and we claim that $\partial\alpha'=\mathcal{K}\sqcup\mathcal{K}'$. To see this, consider any oriented $d$-simplex $\Sigma$ on the vertices of $\mathcal{K}\sqcup\mathcal{K}'$ and the number of times it appears in $\partial \alpha'$:
\begin{enumerate}
    \item If the vertices of $\Sigma$ are not either all in $\mathcal{K}$ or all in $\mathcal{K}'$, then $\Sigma$ occurs 0 times in both $\partial\alpha$ (because $\partial\alpha=\mathcal{K}\sqcup\mathcal{K}'$) and in $\partial \alpha'$ (by construction).
    \item If the vertices of $\Sigma$ are all in $\mathcal{K}$ and do not include $v$, then it is clear that $\Sigma$ has the same coefficient in $\partial\alpha$ and $\partial\alpha'$.
    \item If the vertices of $\Sigma$ are all in $\mathcal{K}$ and include a copy of $v$, then we can distinguish the copies of $\Sigma$ in $\partial\alpha'$ by their provenance in $\partial\alpha$. Some copies of $\Sigma$ come directly from copies of $\Sigma$ in $\partial\alpha$, and the rest result from Step 2 in the formation of $\alpha'$. We will show these latter copies resulting from Step 2 have no net contribution by a graph-theoretic argument.
    
    Construct digraph $G$ with vertex set the possible oriented $d$-simplices which would change to $\Sigma$ under Step 2 (i.e. $\Sigma$, with a $v$ replaced by a vertex from $\mathcal{K}'$). Call these graph vertices $\{F_i\}$. Also add a special vertex $\star$ to $G$. Each oriented $(d+1)$-simplex in $\alpha$ with a boundary $F_i$ (or it's negative) falls into 2 cases: (1) the kind relevant for Step 2 above, with $d+1$ vertices in $\mathcal{K}$, $1$ vertex in $\mathcal{K}$'; and (2) the kind which is ignored in the construction of $\alpha'$, with $d$ vertices in $\mathcal{K}$ and 2 vertices in $\mathcal{K}'$. For each type (1) $(d+1)$-simplex $\sigma$ we add directed edges between $\star$ and the contained $F_i$, with edge \textit{to} the $F_i$ if the $F_i$ appears on the boundary of $\sigma$ with positive orientation, \textit{away} otherwise. For each type (2) $(d+1)$-simplex $\sigma$ we add a directed edge between the two $F_i$ on its boundary, from the one with negative orientation to the one with positive orientation.
    
    Notice that, in $G$, the in-degrees and out-degrees of each $F_i$ must be equal, because $\partial\alpha$ has no triangles with vertices split between $\mathcal{K}$ and $\mathcal{K}'$. Notice also that the difference between the in-degree and out-degree of $\star$ is exactly the net contribution to $\Sigma$ resulting from Step 2. But by a divergence-theorem style argument we see that the difference of the total number of in-degrees and out-degrees among the $F_i$ ---namely zero --- is exactly the difference of the in-degree and out-degree of $\star$. Thus $\star$ must also have equal in-degree and out-degree and so we see the Step 2 copies of $\Sigma$ have no net contributions. 
    
    Thus $\Sigma$ has the same number of copies in $\partial \alpha$ and $\partial\alpha'$, completing this case
    \end{enumerate}
The remaining two cases are simply the previous two on the $\mathcal{K}'$ side.

We thus have that $\partial\alpha'=\mathcal{K}\sqcup\mathcal{K}'$, and, since $\alpha'$ has no $(d+1)$-simplices between $\mathcal{K}$ and $\mathcal{K}'$, we have $V_\mathbb{Z}(\mathcal{K})+V_\mathbb{Z}(\mathcal{K}')\leq|\alpha'|$. Combining this with $|\alpha'|_1\leq |\alpha|_1=V_\mathbb{Z}(\mathcal{K}\sqcup\mathcal{K}')$ establishes our lower bound and completes the proof.
\end{proof}
This proof yields two interesting corollaries.
\begin{cor}
    Let $\mathcal{K}$, $\mathcal{K}'$ be admissible $d$-complexes, and let $\mathcal{K}\sqcup\mathcal{K}'$ their disjoint union. Then there exists an integral $(d+1)$-chain optimally decomposing $\mathcal{K}\sqcup\mathcal{K}$ (i.e. $\partial\alpha=\mathcal{K}\sqcup\mathcal{K}'$, $|\alpha|_1=V_\mathbb{Z}(\mathcal{K}\sqcup\mathcal{K}')$) such that every simplex of $\alpha$ (with non-zero coefficient) is either entirely on the vertices of $\mathcal{K}$ or entirely on the vertices of $\mathcal{K}'$.
\end{cor}
\begin{proof}
From the proof, simply take $\alpha'$.
\end{proof}

\begin{cor}
    Let $\mathcal{K}$, $\mathcal{K}'$ be admissible $d$-complexes, $\mathcal{K}\sqcup\mathcal{K}'$ their disjoint union, and $\alpha$ any integral $(d+1)$-chain optimally decomposing $\mathcal{K}\sqcup\mathcal{K}$ (i.e. $\partial\alpha=\mathcal{K}\sqcup\mathcal{K}'$, $|\alpha|_1=V_\mathbb{Z}(\mathcal{K}\sqcup\mathcal{K}')$).
    Then every simplex of $\alpha$ (with non-zero coefficient) must fall into one of the following cases:
    \begin{itemize}
        \item entirely on the vertices of $\mathcal{K}$
        \item entirely on the vertices of $\mathcal{K}'$
        \item entirely on the vertices of $\mathcal{K}$ except for one vertex in $\mathcal{K}'$
        \item entirely on the vertices of $\mathcal{K}'$ except for one vertex in $\mathcal{K}$.
    \end{itemize}.
\end{cor}
\begin{proof}
Examining the inequalities in the proof of Theorem \ref{thm:Z disjoint union}, we see that we must actually have $|\alpha'|_1=|\alpha|_1$.
\end{proof}

The following proof establishes that $V_\mathbb{Z}$ is also additive under ``$d$-simplex'' connected sum for admissible $d$-dimensional complexes. 

\begin{thm}\label{thm: Z connected sum}
    Let $\mathcal{K},\mathcal{K}'$ be admissible $d$-complexes, $F=[v_0,...,v_k],F'=[v'_0,...,v'_k]$ be respective facets, and form the connected sum $\mathcal{K}\#\mathcal{K}'$ by starting with $\mathcal{K}\sqcup\mathcal{K}'$, then identifying $v_0,v_0';\ v_1,v_k';...; \ v_k,v_1'$ and deleting $F,F'$. Then 
    $$V_\mathbb{Z}(\mathcal{K}\#\mathcal{K}') = V_\mathbb{Z}(\mathcal{K})+V_\mathbb{Z}(\mathcal{K}').$$
\end{thm}
\begin{proof}
We proceed by essentially the same argument as Theorem \ref{thm:Z disjoint union} for disjoint union. We will refer to the identified pairs as $v_0/v_0'$, etc, and call such pairs `central vertices', the set of which we denote $C$.\\
Clearly we have $V_\mathbb{Z}(\mathcal{K}\#\mathcal{K'})\leq V_\mathbb{Z}(\mathcal{K})+V_\mathbb{Z}(\mathcal{K}')$ since we can combine optimal decompositions of $\mathcal{K},\mathcal{K}'$ to obtain a candidate decomposition of $\mathcal{K}\sqcup\mathcal{K'}$ --- simply replace instances of the $v_i$ or $v_i'$ by their identified pairs. \\
For the reverse inclusion, suppose we have a $(d+1)$-chain $\alpha$ realizing $V_\mathbb{Z}(\mathcal{K}\#\mathcal{K}')$. We obtain a new chain $\alpha'$ as follows:
\begin{enumerate}
    \item For each term $c_i T_i$ of $\alpha$, where the vertices of $V_i$ are either entirely in $\mathcal{K}$ (including $C$) or entirely in $\mathcal{K}'$ (including $C$), add this term to $\alpha'$ without modification.
    \item For each term $c_i T_i$ of $\alpha$, where the $T_i$ has all but one vertex in $\mathcal{K}$ (including $C$), contribute the term $c_i \tilde{T}_i$ to $\alpha'$ --- where $\tilde{T}$ is obtained by replacing the vertex of $T$ in $\mathcal{K'}$ by $v_0/v_0'$.
    \item For each term $c_i T_i$ of $\alpha$, where the $T_i$ has all but one vertex in $\mathcal{K}'$ (including $C$), contribute the term $c_i \tilde{T}_i$ to $\alpha'$ --- where $\tilde{T}$ is obtained by replacing the vertex of $T$ in $\mathcal{K}$ by $v_1/v_1'$.
    \item Other terms of $\alpha$ are ignored, i.e. no contribution to $\alpha'$.
\end{enumerate}
Clearly we have $|\alpha'|_1\leq |\alpha|_1=V_\mathbb{Z}(\mathcal{K}\#\mathcal{K}')$, and we claim that $\partial\alpha'=\mathcal{K}\#\mathcal{K}'$. To see this, consider any oriented $d$-simplex $\Sigma$ on the vertices of $\mathcal{K}\#\mathcal{K}'$ and the number of times it appears in $\partial \alpha'$:
\begin{enumerate}
    \item If the vertices of $\Sigma$ are not either all in $\mathcal{K}$ (including $C$) or all in $\mathcal{K}'$ (including $C$), then $\Sigma$ occurs 0 times in both $\partial\alpha$ (because $\partial\alpha=\mathcal{K}\#\mathcal{K}'$) and in $\partial \alpha'$ (by construction).
    \item If the vertices of $\Sigma$ are all in $\mathcal{K}$ (including $C$) and do not include $v_0/v_0'$, then it is clear that $\Sigma$ has the same coefficient in $\partial\alpha$ and $\partial\alpha'$.
    \item Similarly, if the vertices of $\Sigma$ are all in $\mathcal{K'}$ (including $C$) and do not include $v_1/v_1'$, then it is clear that $\Sigma$ has the same coefficient in $\partial\alpha$ and $\partial\alpha'$.
    \item Now consider the case where the vertices of $\Sigma$ are all in $\mathcal{K}$ (including $C$) and include a copy of $v_0/v_0'$, but not $v_1/v_1'$. We address this case by graph-theoretic argument similar to the disjoint-union case. Like before, it suffices to show that the copies of $\Sigma$ arising from Step 2 in the formation of $\alpha$ have no net contribution.\\
    Construct digraph $G$ with vertex set the possible oriented $d$-simplices which would change to $\Sigma$ under Step 2 (i.e. $\Sigma$, with a $v_0/v_0'$ replaced by a vertex from $\mathcal{K}'$). Call these graph vertices $\{F_i\}$. Also add a special vertex $\star$ to $G$. Each oriented $(d+1)$-simplex in $\alpha$ with a boundary $F_i$ (or it's negative) falls into 2 cases: (1) the kind relevant for Step 2 above, with $d+1$ vertices in $\mathcal{K}$ (including $C$), $1$ vertex in $\mathcal{K}'\setminus C$; and (2) the kind which is ignored in the construction of $\alpha'$, with $d$ vertices in $\mathcal{K}$ (including $C$) and 2 vertices in $\mathcal{K}'\setminus C$. For each type (1) $(d+1)$-simplex $\sigma$ we add directed edges between $\star$ and the contained $F_i$, with edge \textit{to} the $F_i$ if the $F_i$ appears on the boundary of $\sigma$ with positive orientation, \textit{away} otherwise. For each type (2) $(d+1)$-simplex $\sigma$ we add a directed edge between the two $F_i$ on its boundary, from the one with negative orientation to the one with positive orientation.
    
    Notice that, in $G$, the in-degrees and out-degrees of each $F_i$ must be equal, because $\partial\alpha$ has no triangles with vertices split between $\mathcal{K}$ and $\mathcal{K}'$. Notice also that the difference between the in-degree and out-degree of $\star$ is exactly the net contribution to $\Sigma$ resulting from Step 2. But by a divergence-theorem style argument we see that the difference of the total number of in-degrees and out-degrees among the $F_i$ ---namely zero --- is exactly the difference of the in-degree and out-degree of $\star$. Thus $\star$ must also have equal in-degree and out-degree and so we see the Step 2 copies of $\Sigma$ have no net contributions. 
    
    Thus $\Sigma$ has the same number of copies in $\partial \alpha$ and $\partial\alpha'$, completing this case.
    
    \item The case where the vertices of $\Sigma$ are all in $\mathcal{K'}$ (including $C$) and include a copy of $v_1/v_1'$, but not $v_0/v_0'$ proceeds by exactly the same argument as the previous case.

    \item The final case is when the vertices of $\Sigma$ are exactly the vertices of $C$. But this case is actually simple since copies of $\Sigma$ are not created or destroyed in the process of forming $\alpha'$.
    \end{enumerate}
This establishes that $\alpha$ and $\alpha'$ have the same boundary, namely $\mathcal{K}\#\mathcal{K}'$. Next we note that the $(d+1)$-simplices of $\alpha'$ on the vertices of $\mathcal{K}$ (including $C$) have boundary $\mathcal{K}$. This follows from the proof since the boundary must be exactly $\mathcal{K}$ except possibly for the number of copies of the $d$-simplices on the central vertices $C$. But because the boundary of a boundary must be zero, and $\partial K=0$, it's not hard to see we must have exactly one copy of $[v_0,...,v_k]$ also. Similarly, the $(d+1)$-simplices of $\alpha'$ on the vertiecs of $\mathcal{K}'$ (including $C$) have boundary $\mathcal{K'}$.\\
Thus we must have $V_\mathbb{Z}(\mathcal{K})+V_\mathbb{Z}(\mathcal{K}')\leq|\alpha'|$. Combining this with $|\alpha'|_1\leq |\alpha|_1=V_\mathbb{Z}(\mathcal{K}\#\mathcal{K}')$ establishes our lower bound and completes the proof.
\end{proof}
We also have the analogous corollaries to the disjoint union case.
\begin{cor}
     Let $\mathcal{K},\mathcal{K}'$ be admissible $d$-complexes, $F=[v_0,...,v_k],F'=[v'_0,...,v'_k]$ be respective facets, and form the connected sum $\mathcal{K}\#\mathcal{K}'$ by starting with $\mathcal{K}\sqcup\mathcal{K}'$, then identifying $v_0,v_0';\ v_1,v_k';...; \ v_k,v_1'$ and deleting $F,F'$. Then there exists an integral $(d+1)$-chain optimally decomposing $\mathcal{K}\#\mathcal{K}$ (i.e. $\partial\alpha=\mathcal{K}\sqcup\mathcal{K}'$, $|\alpha|_1=V_\mathbb{Z}(\mathcal{K}\sqcup\mathcal{K}')$) such that every simplex of $\alpha$ (with non-zero coefficient) is either entirely on the vertices of $\mathcal{K}$ (including central vertices) or entirely on the vertices of $\mathcal{K}'$ (including central vertices).
\end{cor}
\begin{proof}[]
\end{proof}

\begin{cor}
     Let $\mathcal{K},\mathcal{K}'$ be admissible $d$-complexes, $F=[v_0,...,v_k],F'=[v'_0,...,v'_k]$ be respective facets, and form the connected sum $\mathcal{K}\#\mathcal{K}'$ by starting with $\mathcal{K}\sqcup\mathcal{K}'$, then identifying $v_0,v_0';\ v_1,v_k';...; \ v_k,v_1'$ and deleting $F,F'$. Let $\alpha$ be any integral $(d+1)$-chain optimally decomposing $\mathcal{K}\#\mathcal{K}$ (i.e. $\partial\alpha=\mathcal{K}\#\mathcal{K}'$, $|\alpha|_1=V_\mathbb{Z}(\mathcal{K}\#\mathcal{K}')$).
    Then every simplex of $\alpha$ (with non-zero coefficient) must fall into one of the following cases:
    \begin{itemize}
        \item entirely on the vertices of $\mathcal{K}$ (including central vertices)
        \item entirely on the vertices of $\mathcal{K}'$ (including central vertices)
        \item entirely on the vertices of $\mathcal{K}$ (including central vertices) except for one vertex in $\mathcal{K}'$
        \item entirely on the vertices of $\mathcal{K}'$ (including central vertices) except for one vertex in $\mathcal{K}$.
    \end{itemize}
\end{cor}
\begin{proof}[]
\end{proof}

\section{An integrality gap in dimension 1}\label{sec_ig_1d}
\begin{figure}[p]
    \centering
    \tikzmath{\s = 3; \sf = .1; \eo = 1.5; \es = \s * (sqrt(3)/2); \les = \s / 2; \al =.92; \ah = 1.4; \at = (\al + \ah) / 2 + .05; \ta = 4; \tal = 9;}
\begin{tikzpicture}
    \foreach \x in {1,3,5}
        {
        \draw [<-] ({\s * cos(\x * 60)+\sf * (\s * cos((\x + 1) * 60) - (\s * cos(\x * 60)))}, {\s * sin(\x * 60) + \sf *((\s * sin((\x + 1) * 60))-(\s * sin(\x * 60)))}) -- 
        ({\s * cos((\x + 1) * 60) + \sf * ((\s * cos(\x * 60))-(\s * cos((\x + 1) * 60)))}, {\s * sin((\x + 1) * 60)+\sf  * ((\s * sin(\x * 60)) - (\s * sin((\x + 1) * 60)))});
        \draw[->] ({(1-\sf)*\s * cos(\x * 60)}, {(1-\sf)*\s * sin(\x * 60)}) -- ({\sf * \s * cos((\x) * 60)}, {\sf * \s * sin((\x) * 60)});
        \draw[<-] ({1+\sf)*\s*cos(\x * 60)}, {(1+\sf)*\s * sin(\x * 60)}) -- ({\eo * \s * cos((\x) * 60)}, {\eo * \s * sin((\x) * 60)});
        }
    \foreach \x in {0,2,4}
        {
        \draw [<-] ({\s * cos(\x * 60)+\sf * (\s * cos((\x + 1) * 60) - (\s * cos(\x * 60)))}, {\s * sin(\x * 60) + \sf *((\s * sin((\x + 1) * 60))-(\s * sin(\x * 60)))}) -- 
        ({\s * cos((\x + 1) * 60) + \sf * ((\s * cos(\x * 60))-(\s * cos((\x + 1) * 60)))}, {\s * sin((\x + 1) * 60)+\sf  * ((\s * sin(\x * 60)) - (\s * sin((\x + 1) * 60)))});
        \draw[<-] ({(1-\sf)*\s * cos(\x * 60)}, {(1-\sf)*\s * sin(\x * 60)}) -- ({\sf * \s * cos((\x) * 60)}, {\sf * \s * sin((\x) * 60)});
        \draw[->] ({1+\sf)*\s*cos(\x * 60)}, {(1+\sf)*\s * sin(\x * 60)}) -- ({\eo * \s * cos((\x) * 60)}, {\eo * \s * sin((\x) * 60)});
        }
    \node at (0,0) {A};
    \node at ({\s * cos(0 * 60)}, {\s * sin(0 * 60)}) {D};
    \node at ({\s * cos(1 * 60)}, {\s * sin(1 * 60)}) {C};
    \node at ({\s * cos(2 * 60)}, {\s * sin(2 * 60)}) {B};
    \node at ({\s * cos(3 * 60)}, {\s * sin(3 * 60)}) {G};
    \node at ({\s * cos(4 * 60)}, {\s * sin(4 * 60)}) {F};
    \node at ({\s * cos(5 * 60)}, {\s * sin(5 * 60)}) {E};

    \node at ({(\eo +\sf+.05) * \s * cos(0 * 60)}, {(\eo +\sf+.05)*\s * sin(0 * 60)}) {\footnotesize{\textit{to G}}};
    \node at ({(\eo +\sf+.05) * \s * cos(1 * 60)}, {(\eo +\sf+.05)*\s * sin(1 * 60)}) {\footnotesize{\textit{from F}}};
    \node at ({(\eo +\sf+.05) * \s * cos(2 * 60)}, {(\eo +\sf+.05)*\s * sin(2 * 60)}) {\footnotesize{\textit{to E}}};
    \node at ({(\eo +\sf+.1) * \s * cos(3 * 60)}, {(\eo +\sf+.1)*\s * sin(3 * 60)}) {\footnotesize{\textit{from D}}};
    \node at ({(\eo +\sf+.05) * \s * cos(4 * 60)}, {(\eo +\sf+.05)*\s * sin(4 * 60)}) {\footnotesize{\textit{to C}}};
    \node at ({(\eo +\sf+.05) * \s * cos(5 * 60)}, {(\eo +\sf+.05)*\s * sin(5 * 60)}) {\footnotesize{\textit{from B}}};
\end{tikzpicture}
    \caption{The simplest (we believe) example with an integrality gap in dimension 1}
    \label{fig:1dgapmain}
    \tikzmath{\s = 2; \sf = .1; \eo = 1.5; \es = \s * (sqrt(3)/2); \les = \s / 2; \al =.92; \ah = 1.4; \at = (\al + \ah) / 2 + .05; \ta = 4; \tal = 9;}
\begin{tikzpicture}
    \foreach \x in {1,3,5}
        {
        \draw [<-] ({\s * cos(\x * 60)+\sf * (\s * cos((\x + 1) * 60) - (\s * cos(\x * 60)))}, {\s * sin(\x * 60) + \sf *((\s * sin((\x + 1) * 60))-(\s * sin(\x * 60)))}) -- 
        ({\s * cos((\x + 1) * 60) + \sf * ((\s * cos(\x * 60))-(\s * cos((\x + 1) * 60)))}, {\s * sin((\x + 1) * 60)+\sf  * ((\s * sin(\x * 60)) - (\s * sin((\x + 1) * 60)))});
        \draw[->] ({(1-\sf)*\s * cos(\x * 60)}, {(1-\sf)*\s * sin(\x * 60)}) -- ({\sf * \s * cos((\x) * 60)}, {\sf * \s * sin((\x) * 60)}) node[midway,above] {2};
        \draw[<-] ({1+\sf)*\s*cos(\x * 60)}, {(1+\sf)*\s * sin(\x * 60)}) -- ({\eo * \s * cos((\x) * 60)}, {\eo * \s * sin((\x) * 60)});
        }
    \foreach \x in {0,2,4}
        {
        \draw[<-] ({(1-\sf)*\s * cos(\x * 60)}, {(1-\sf)*\s * sin(\x * 60)}) -- ({\sf * \s * cos((\x) * 60)}, {\sf * \s * sin((\x) * 60)}) node[midway,above] {2};
        \draw[->] ({1+\sf)*\s*cos(\x * 60)}, {(1+\sf)*\s * sin(\x * 60)}) -- ({\eo * \s * cos((\x) * 60)}, {\eo * \s * sin((\x) * 60)});
        }
    \node at (0,0) {A};
    \node at ({\s * cos(0 * 60)}, {\s * sin(0 * 60)}) {D};
    \node at ({\s * cos(1 * 60)}, {\s * sin(1 * 60)}) {C};
    \node at ({\s * cos(2 * 60)}, {\s * sin(2 * 60)}) {B};
    \node at ({\s * cos(3 * 60)}, {\s * sin(3 * 60)}) {G};
    \node at ({\s * cos(4 * 60)}, {\s * sin(4 * 60)}) {F};
    \node at ({\s * cos(5 * 60)}, {\s * sin(5 * 60)}) {E};

    \node at ({(\eo +\sf+.1) * \s * cos(0 * 60)}, {(\eo +\sf+.05)*\s * sin(0 * 60)}) {\footnotesize{\textit{to G}}};
    \node at ({(\eo +\sf+.05) * \s * cos(1 * 60)}, {(\eo +\sf+.05)*\s * sin(1 * 60)}) {\footnotesize{\textit{from F}}};
    \node at ({(\eo +\sf+.05) * \s * cos(2 * 60)}, {(\eo +\sf+.05)*\s * sin(2 * 60)}) {\footnotesize{\textit{to E}}};
    \node at ({(\eo +\sf+.2) * \s * cos(3 * 60)}, {(\eo +\sf+.2)*\s * sin(3 * 60)}) {\footnotesize{\textit{from D}}};
    \node at ({(\eo +\sf+.05) * \s * cos(4 * 60)}, {(\eo +\sf+.05)*\s * sin(4 * 60)}) {\footnotesize{\textit{to C}}};
    \node at ({(\eo +\sf+.05) * \s * cos(5 * 60)}, {(\eo +\sf+.05)*\s * sin(5 * 60)}) {\footnotesize{\textit{from B}}};
\end{tikzpicture}
    \caption{1-D integrality gap, reference 1, unmarked edges have multiplicity 1}
    \label{fig:1dgap1}
    \tikzmath{\s = 2; \sf = .1; \eo = 1.5; \es = \s * (sqrt(3)/2); \les = \s / 2; \al =.92; \ah = 1.4; \at = (\al + \ah) / 2 + .05; \ta = 4; \tal = 9;}
\begin{tikzpicture}
    \foreach \x in {1,3,5}
        {
        \draw [<-] ({\s * cos(\x * 60)+\sf * (\s * cos((\x + 1) * 60) - (\s * cos(\x * 60)))}, {\s * sin(\x * 60) + \sf *((\s * sin((\x + 1) * 60))-(\s * sin(\x * 60)))}) -- 
        ({\s * cos((\x + 1) * 60) + \sf * ((\s * cos(\x * 60))-(\s * cos((\x + 1) * 60)))}, {\s * sin((\x + 1) * 60)+\sf  * ((\s * sin(\x * 60)) - (\s * sin((\x + 1) * 60)))});
        \draw[<-] ({1+\sf)*\s*cos(\x * 60)}, {(1+\sf)*\s * sin(\x * 60)}) -- ({\eo * \s * cos((\x) * 60)}, {\eo * \s * sin((\x) * 60)});
        }
    \foreach \x in {0,2,4}
        {
        \draw [<-] ({\s * cos(\x * 60)+\sf * (\s * cos((\x + 1) * 60) - (\s * cos(\x * 60)))}, {\s * sin(\x * 60) + \sf *((\s * sin((\x + 1) * 60))-(\s * sin(\x * 60)))}) -- 
        ({\s * cos((\x + 1) * 60) + \sf * ((\s * cos(\x * 60))-(\s * cos((\x + 1) * 60)))}, {\s * sin((\x + 1) * 60)+\sf  * ((\s * sin(\x * 60)) - (\s * sin((\x + 1) * 60)))}) node[midway,above] {2};
        \draw[->] ({1+\sf)*\s*cos(\x * 60)}, {(1+\sf)*\s * sin(\x * 60)}) -- ({\eo * \s * cos((\x) * 60)}, {\eo * \s * sin((\x) * 60)});
        }
    \node at ({\s * cos(0 * 60)}, {\s * sin(0 * 60)}) {D};
    \node at ({\s * cos(1 * 60)}, {\s * sin(1 * 60)}) {C};
    \node at ({\s * cos(2 * 60)}, {\s * sin(2 * 60)}) {B};
    \node at ({\s * cos(3 * 60)}, {\s * sin(3 * 60)}) {G};
    \node at ({\s * cos(4 * 60)}, {\s * sin(4 * 60)}) {F};
    \node at ({\s * cos(5 * 60)}, {\s * sin(5 * 60)}) {E};

    \node at ({(\eo +\sf+.1) * \s * cos(0 * 60)}, {(\eo +\sf+.05)*\s * sin(0 * 60)}) {\footnotesize{\textit{to G}}};
    \node at ({(\eo +\sf+.05) * \s * cos(1 * 60)}, {(\eo +\sf+.05)*\s * sin(1 * 60)}) {\footnotesize{\textit{from F}}};
    \node at ({(\eo +\sf+.05) * \s * cos(2 * 60)}, {(\eo +\sf+.05)*\s * sin(2 * 60)}) {\footnotesize{\textit{to E}}};
    \node at ({(\eo +\sf+.2) * \s * cos(3 * 60)}, {(\eo +\sf+.2)*\s * sin(3 * 60)}) {\footnotesize{\textit{from D}}};
    \node at ({(\eo +\sf+.05) * \s * cos(4 * 60)}, {(\eo +\sf+.05)*\s * sin(4 * 60)}) {\footnotesize{\textit{to C}}};
    \node at ({(\eo +\sf+.05) * \s * cos(5 * 60)}, {(\eo +\sf+.05)*\s * sin(5 * 60)}) {\footnotesize{\textit{from B}}};
\end{tikzpicture}
    \caption{1-D integrality gap, reference 2, unmarked edges have multiplicity 1}
    \label{fig:1dgap2}
\end{figure}
Computer search yielded the complex depicted in Figure \ref{fig:1dgapmain}, with $V_\mathbb{Z}=7$, $V_\mathbb{Q}=6$. These value of $V_\mathbb{Z}$ were computed by the method described in Section \ref{lp_ip_section}. We believe this complex exhibits the simplest example of an integrality gap.

To decompose it into 7 (integral) triangles, one can take the rings through A ($ABE$, $CAF$, $GAD$), together with a decomposition of the boundary hexagon $BCDEFG$ into four triangles.

To see it can decomposed into 6 fractional triangles, note that it suffices to decompose a doubled version, with two copies of each edge, into 12 integer triangles (because then we can give all the integer triangles weight 1/2). And we can indeed do this by combining decompositions of Figure \ref{fig:1dgap1} with Figure \ref{fig:1dgap2} into 6 integer triangles each. Figure \ref{fig:1dgap1} may be decomposed via the rings ($ABE$, $CAF$, $GAD$), together with $ABC$, $ADE$, $AFG$. Figure \ref{fig:1dgap2} may be decomposed as $GBC,\ GCD,\ CDE,\ CEF,\ EFG,\ EGB$. Combining these decompositions and halving gives the desired decomposition into 6 fractional triangles. 

\section{Integrality gaps in triangulations of the 2-sphere}\label{sec_ig_2d}
Using Plantri\footnote{\url{https://users.cecs.anu.edu.au/~bdm/plantri/}, see also \cite{plantri}}, we enumerated triangulations of the 2-sphere and computed $V_\mathbb{Q}$ and $V_\mathbb{Z}$ for each using the technique of Section \ref{lp_ip_section}. To cut down this large search space, \textit{we restricted attention to triangulations of the sphere where each vertex has max degree 6}.

Our results are as follows:
\begin{enumerate}
    \item For $4\leq n\leq 16$ vertices, $V_\mathbb{Q}=V_\mathbb{Z}$.
    \item The first integrality gap appears for $n=17$, where there are two examples with a gap of $.5$, \mbox{both with $V_\mathbb{Q}=23.5,\ V_\mathbb{Z}=24$}. These shapes are are pictured in Figure \ref{fig:n17_gap_polys}, which contains links to interactive models.
\begin{figure}[p]
    \centering
    \includegraphics[width=.3569\linewidth]{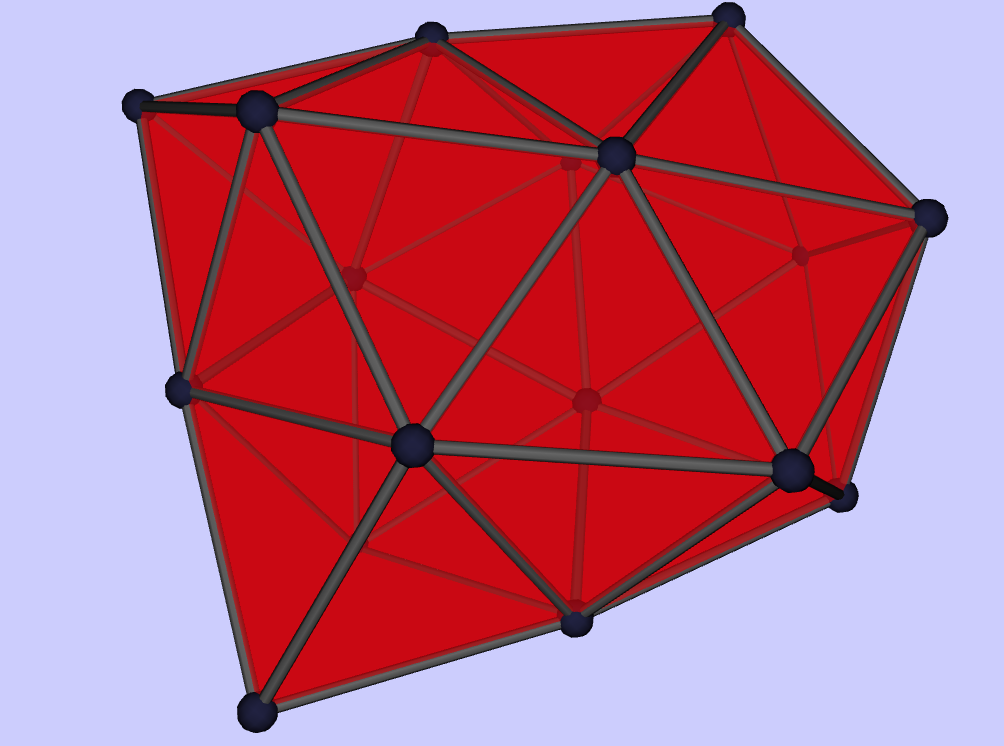}
    \includegraphics[width=.4\linewidth]{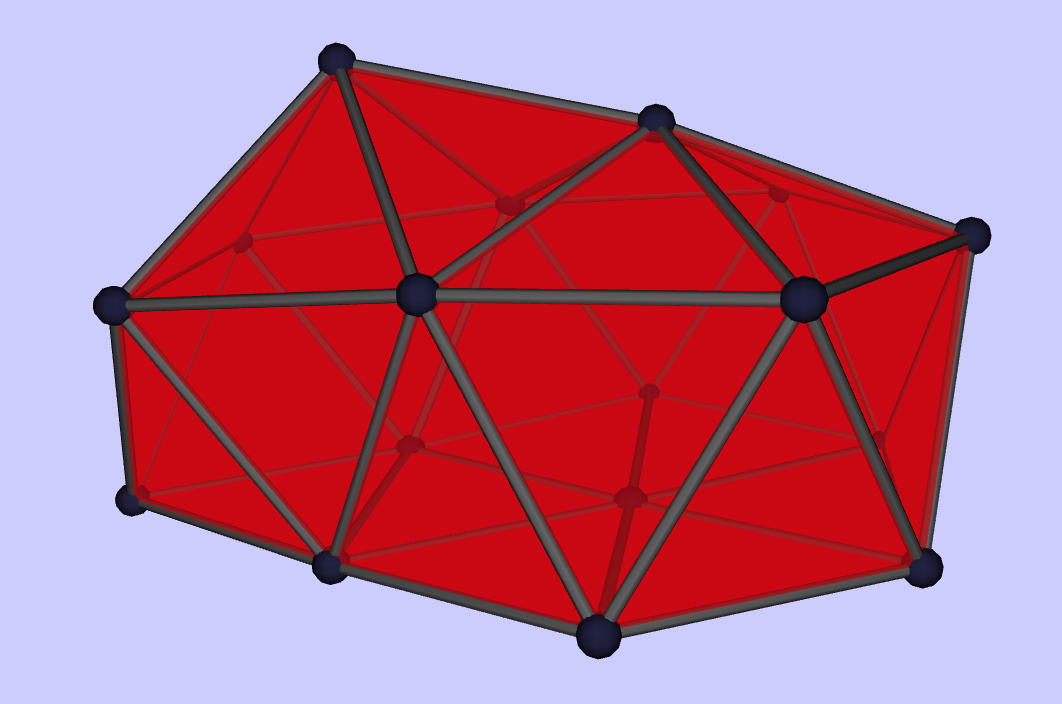}
    \caption{The two triangulations on 17 vertices with integrality gap $1/2$ (models: \href{https://math.dartmouth.edu/~mellison/tripolys/17\_vertices/17\_226/}{left}, \href{https://math.dartmouth.edu/~mellison/tripolys/17\_vertices/17\_181/}{right})\\
    \tiny{Left plantri string$^*$: bcdefg aghic abid acijke adklf aelg aflmhb bgmnoi bhojdc diopk djpqle ekqmgf glqnh hmqpo hnpji jonqk kpnml}\\
    \tiny{Right plantri string: bcdefg aghijc abjkd ackle adlmnf aenog afohb bgopqi bhqlkj bikc cjild dkiqme elqpn empof fnphg honmq hpmli}\\\\
    \tiny{$^*$ To interpret this, we label the 17 vertices by the letters `a' through `q'. The neighboring vertices to `a' are (in counterclockwise order, when viewed from outside) `bcdefg', and so on.}}
    \label{fig:n17_gap_polys}
    \vspace{2cm}
    \includegraphics[width=.54\linewidth]{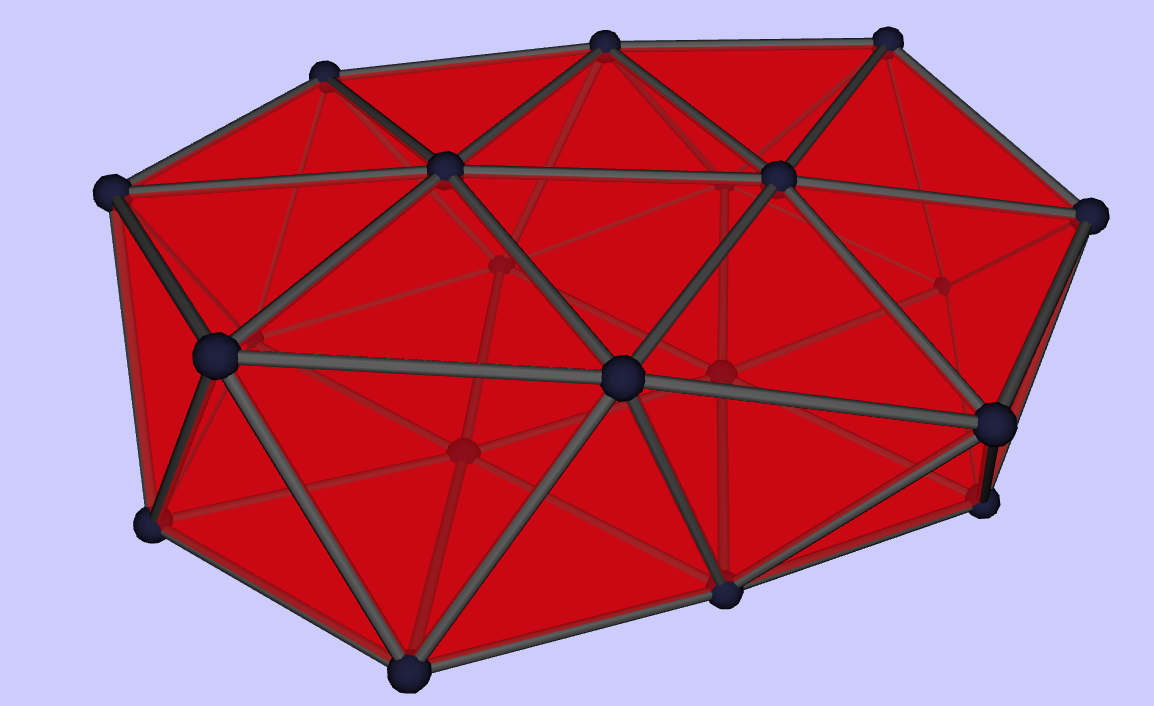}
    \includegraphics[width=.4\linewidth]{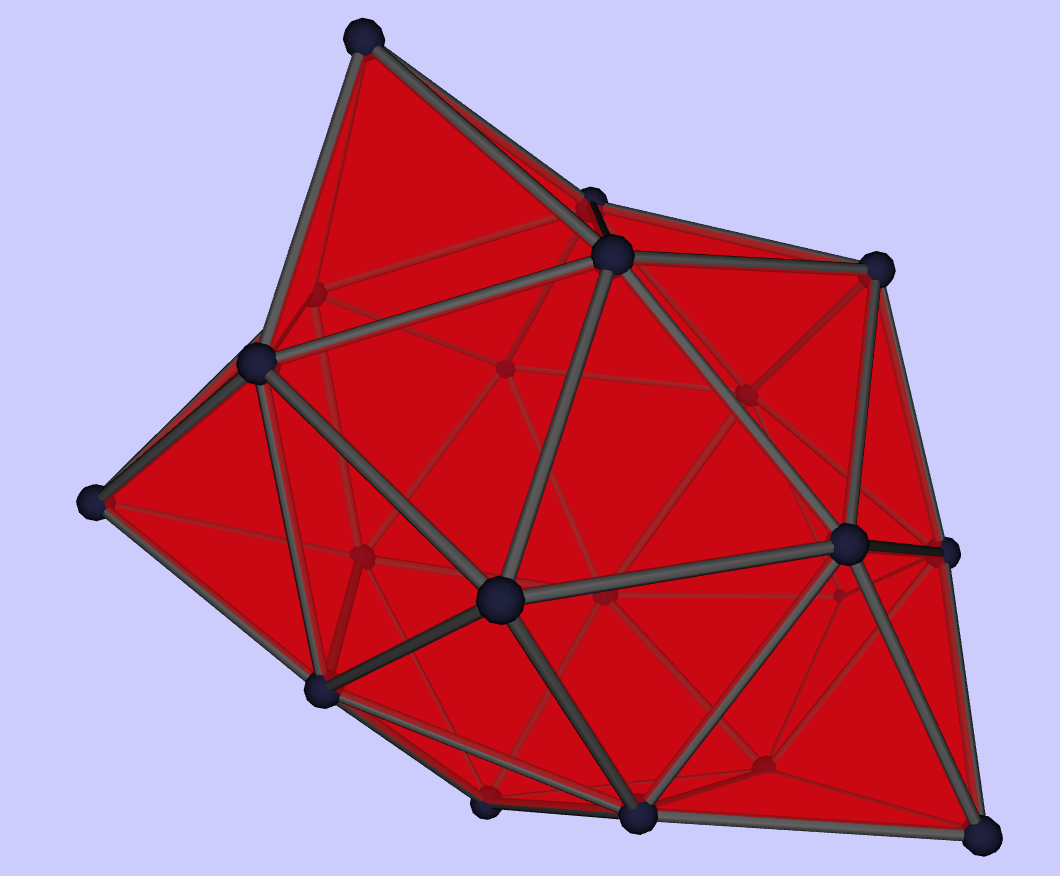}
     \caption{Two triangulations on 20 vertices with gaps $1/5$, left, and $1/3$, right (models: \href{https://math.dartmouth.edu/~mellison/tripolys/20\_vertices/20\_747/}{left}, \href{https://math.dartmouth.edu/~mellison/tripolys/20\_vertices/20\_834/}{right})\\
    \tiny{Left plantri string: bcdefg aghic abijd acjke adklmf aemnog afophb bgpqri bhrjc cirkd djrle ekrqsm elsnf fmsto fntpg gotqh hptslr hqlkji lqtnm nsqpo\\
    \tiny{Right plantri string: bcdefg aghc abhijd acje adjklf aelmg afmnhb bgnoic chopj cipked ejpql ekqrmf flrsng gmsoh hnstpi iotqkj kptsrl lqsm mrqton osqp}}}
     \label{fig:n20_gap_polys}
\end{figure}
    
    \item The first gaps not equal to $1/2$ are $1/5$ and $1/3$, both appearing at $n=20$. One shape with gap $1/5$ is easy to describe as an antiprism of an extended hexagon, the other is not so easy to describe, see Figure \ref{fig:n20_gap_polys}, which contains links to interactive models of both.

\end{enumerate}
Our computations have not systematically progressed beyond $n=20$. The largest gap we have found for this family of triangulations is $2/3$ for a shape on 28 vertices, see also \cite{doyle2023extending}).

It would be great to establish some bounds on the integrality gap for triangulations of the sphere with maximum degree 6 --- as far as we know $\lceil V_\mathbb{Q}\rceil = V_\mathbb{Z}$, which would be a very interesting result if true!\\
We note that in cases where $\lceil V_\mathbb{Q}(\mathcal{K})\rceil = V_\mathbb{Z}(\mathcal{K})$ we can obtain a decomposition achieving $V_\mathbb{Z}$ in polynomial time. Recall that integer programming is generally NP-hard. So, if the statement holds for max-degree-6 triangulations of the sphere, the following procedure computes $V_\mathbb{Z}$ in polynomial time:
\vbox{
\vspace{.15cm}
\hrule
\vspace{.15cm}
For each $(d+1)$-simplex $\tau$ on the vertices of $\mathcal{K}$:
\begin{enumerate}
    \item Remove the boundary of $\tau$ from $\mathcal{K}$ to obtain $\mathcal{K}-\tau$.
    \item If $\lceil V_\mathbb{Q}(\mathcal{K}-\tau)\rceil + 1 = \lceil V_\mathbb{Q}(\mathcal{K})\rceil$, we use $\tau$ in our decomposition and recurse on $\mathcal{K}-\tau$.
\end{enumerate}
\hrule
}
However $\lceil V_\mathbb{Q}\rceil = V_\mathbb{Z}$ does not hold for general triangulations. Consider taking the disjoint union of two copies of one of the gap-1/2 shapes on 17 vertices. By Theorems \ref{thm: disjoint union} and \ref{thm:Z disjoint union} we find a gap of 1, with $V_{\mathbb{Z}}=48$ and $V_{\mathbb{Q}}=47$, and the gap grows linearly with additional copies. Similarly, by Theorems \ref{thm: connected sum} and \ref{thm: Z connected sum}, the connected sum of two 17-vertex gap 1/2 shapes along a triangle has gap $1$, and this gap grows linearly with additional connected summands.\footnote{Connected sum along 4-cycles is not generally additive for $V_\mathbb{Z}$ and $V_\mathbb{Q}$. We took the connected sum of two copies of a 17-vertex, gap 1/2 shape and found that $V_{\mathbb{Z}}=48$, $V_{\mathbb{Q}}=47$ along a certain 4-cycle, but many 4-cycles yield $V_{\mathbb{Z}}=V_{\mathbb{Q}}=46.$} This shows that the ``LPtet hypothesis" as formulated in \cite{doyle2023extending} (version 1) is false.

\section{Future directions}
In computer science, integrality gaps, such as between $V_\mathbb{Z}$ and $V_\mathbb{Q}$, arise in a number of problems, where bounds are proven to establish guarantees on `approximation algorithms' which return the value of a relaxed program. See for example \cite{intgap},\cite{chakrabarty2009allocating},\cite{haxell2023improved}. In our case, this would amount to bounding $V_\mathbb{Z}/V_\mathbb{Q}$ (as a function of $d$, perhaps number of vertices or facets). As far as we know, this gap appears to not have been considered in the literature, and the best bounds would be an interesting set of constants in the theory of simplicial complexes.

Another direction is to apply the constants $V_\mathbb{Z}$ and $V_\mathbb{Q}$ to yield invariants of known families of admissible $d$-complexes. Admissible 1-complexes, sometimes known as `balanced' digraphs \cite{balanced_digraph}, are one such example. $V_\mathbb{Q}$ can also be applied to obtain invariants of weighted digraphs with equal in-weight and out-weight at each vertex --- or, in other words, doubly stochastic matrices over $\mathbb{Q}$. $V_\mathbb{Z}$ and $V_\mathbb{Q}$ could also be applied to yield invariants of oriented polytopes with simplicial facets (or objects with associated such polytopes, perhaps from combinatorics). Note that, in theory at least, the simplicial facets assumption can be dropped via, for example, a minimum over all decompositions of the facets into simplices.

A third direction might be to take a look at an analog of $V_\mathbb{Z}$ and $V_\mathbb{Q}$ for admissible $d$-complexes with the `oriented' condition dropped. A $V_{\mathbb{F}_2}$ can be defined in this case by taking the norm of a chain over $\mathbb{F}_2$ to be the number of non-zero coefficients.

\bibliographystyle{plain}
\bibliography{refs}

\end{document}